%% file: root.tex
\documentclass[letterpaper, 10pt, conference]{ieeeconf}  

\IEEEoverridecommandlockouts                              

\input{package}

\title{\LARGE \bf
A Distributed Online Convex Optimization Algorithm  \\with Improved Dynamic Regret
}

\author{Yan Zhang, Robert J. Ravier, Michael M. Zavlanos, Vahid Tarokh
\thanks{Yan Zhang and Michael M. Zavlanos are with the Department of Mechanical Engineering and Materials Science, Duke University, Durham, NC 27708, USA. {\tt\small \{yan.zhang2,michael.zavlanos\}@duke.edu} Robert Ravier and Vahid Tarokh are with the Department of Electrical and Computer Engineering, Duke University, Durham, NC 27708, USA. {\tt\small \{robert.ravier,vahid.tarokh\}@duke.edu} This work is supported in part by AFOSR under award \# FA9550-19-1-0169 and by DARPA under grant \# FA8650-18-1-7837.}}

\makeatletter
\newcommand\fs@spaceruled{\def\@fs@cfont{\bfseries}\let\@fs@capt\floatc@ruled
	\def\@fs@pre{\vspace{1.5\baselineskip}\hrule height.8pt depth0pt \kern2pt}%
	\def\@fs@post{\kern2pt\hrule\relax}%
	\def\@fs@mid{\kern2pt\hrule\kern2pt}%
	\let\@fs@iftopcapt\iftrue}
\makeatother

\begin{document}

\maketitle

\thispagestyle{empty}
\pagestyle{empty}

\input{tex/abstract}

\input{tex/intro}

\input{tex/formulation}

\input{tex/convergence}

\input{tex/simulation}

%
\input{tex/conclusion}

\addtolength{\textheight}{-12cm}   

\bibliographystyle{IEEEtran}
\bibliography{biblio}

\end{document}

%% file: package.tex
\usepackage[pass]{geometry}
\usepackage{fancyhdr,todo}
\usepackage{amsmath}
\usepackage{amsfonts,amssymb,float,graphicx}

\usepackage{amsthm}

\usepackage{subfigure}
\usepackage[usenames]{color}
\usepackage{pifont}

\usepackage{multirow}
\usepackage{hhline}

\usepackage{listings}
\usepackage{color} 
\definecolor{mygreen}{RGB}{28,172,0} 
\definecolor{mylilas}{RGB}{170,55,241}

\usepackage{algpseudocode,algorithm,algorithmicx}

\theoremstyle{plain}
\newtheorem{theorem}{Theorem}[section] 

\theoremstyle{definition}
\newtheorem{assumption}[theorem]{Assumption}
\newtheorem{lemma}[theorem]{Lemma}
\newtheorem{remark}[theorem]{Remark}

\usepackage[font=footnotesize,labelfont=bf]{caption}

\usepackage{cite}



%% file: tex/abstract.tex
\begin{abstract}
In this paper, we consider the problem of distributed online convex optimization, where a network of local agents aim to jointly optimize a convex function over a period of multiple time steps. The agents do not have any information about the future. Existing algorithms have established dynamic regret bounds that have explicit dependence on the number of time steps. In this work, we show that we can remove this dependence assuming that the local objective functions are strongly convex. More precisely, we propose a gradient tracking algorithm where agents jointly communicate and descend based on corrected gradient steps. We verify our theoretical results through numerical experiments.
\end{abstract}

%% file: tex/intro.tex
\section{INTRODUCTION}
\label{sec:intro}
Distributed optimization has recently received considerable attention, particularly due to its wide applicability in the areas of control and learning \cite{rabbat2004distributed, ram2009distributed, shahrampour2016distributed, nedic2017fast}. 
The goal is to decompose large optimization problems into smaller, more manageable subproblems that are solved iteratively and in parallel by a group of communicating agents. As such, distributed algorithms avoid the cost and fragility associated with centralized coordination, and provide better privacy for the autonomous decision makers.

In this paper, we are concerned with distributed online convex optimization (OCO) problems, where the objective function can vary with time and the goal is to minimize a notion of {\emph{regret}}. Most effort in OCO has been devoted to analyzing algorithms using {\emph{static}} regret, which compares the computed candidate optima with the best fixed action in hindsight. 
It is well known that the static regret of distributed OCO algorithms grows sublinearly with the iterations; see \cite{zinkevich2003online, hazan2007logarithmic, hosseini2013online, akbari2015distributed, lee2017sublinear} for examples.
Our focus is on {\emph{dynamic}}  regret, which compares the output of the algorithm at each time with the current optimal value at that time. 
Dynamic regret is of more interest when the goal is to track a sequence of varying optimal solutions, as in  distributed tracking of moving targets.
It is well known that, in general, sublinear dynamic regret cannot be achieved even in the centralized case; bounds on dynamic regret 
are related to problem regularities,
which are measures of how the individual optimization problems change over time \cite{hall2015online, jadbabaie2015online, zhang2017improved}. 
The distributed problem has been analyzed in \cite{shahrampour2018distributed,nazari2019dadam}. The work in \cite{shahrampour2018distributed} makes use of an online version of mirror descent, while \cite{nazari2019dadam} uses a distributed variant of the celebrated ADAM algorithm \cite{duchi2011adaptive}. In both cases, the bounds achieved depend on the path-length regularity as well as the total number of time steps.

The dependency of dynamic regret on the number of time steps was removed in \cite{mokhtari2016online} by assuming the objective function to be strongly convex, and then in \cite{zhang2017improved} under weaker conditions than strong convexity, both in the centralized case.
The key idea is that with sufficient assumptions, the standard offline algorithms can be shown to exhibit linear convergence provided a small enough step size is used. 
In this paper, we show that analogous results can hold in the distributed setting by employing the gradient tracking technique in \cite{pu2018push}. Our main result shows that strong convexity allows us to eliminate the explicit dependence on time in the regret bound. The algorithm that we propose is an online modification of those featured in \cite{qu2018harnessing,pu2018push}, which themselves are corrected versions of decentralized gradient descent \cite{nedic2009distributed}. 
In the process, we use another regularity in OCO for variations in the gradient. 

The remainder of the work is organized as follows. In Section \ref{sec:pf}, we formulate the problem and present assumptions necessary for the subsequent analysis. In Section \ref{sec:convergence}, we present our proposed algorithm and analyze its dynamic regret bounds. In Section \ref{sec:simulation} we perform numerical experiments to evaluate the quality of our bounds and algorithm. We make concluding remarks in Section \ref{sec:conclusion}.

%% file: tex/formulation.tex
\section{Review and Preliminaries}
\label{sec:pf}
We briefly review necessary background for the distributed OCO problem at hand. At each time $t$, we are interested in optimizing a differentiable convex function $f_{t}(x):= \sum_{i=1}^{n} f_{i,t}(x),$ where $n$ is the size of the network and $\{f_{i,t}\}$ are differentiable convex functions corresponding to each agent $i$ in a connected network. Let each agent keep a local estimate $x_{i,t}$ of the global optimizer $x_t^\ast$ at $t$. At time $t = 0$, the agents are initialized with starting points $x_{i,0}$ and the gradients $\nabla f_{i,0}(x_{i,0})$ are observed locally. Though many notions of regret exist, there are two notions that have historically been most popular in the OCO literature. The static regret is a measurement of performance with respect to the best fixed decision in hindsight and is given in the distributed setting by
\begin{equation}\label{eqn:static}
R_{T}^{s} := \frac{1}{n} \sum_{i=1}^{n} \sum_{t=0}^{T} f_{t}(x_{i,t}) - \min_{x} \sum_{t=1}^{T} f_{t}(x).
\end{equation}

\noindent The focus of our paper is on dynamic regret. In this measure, the performance of the algorithm is measured with respect to the best decision at each time. Formally, it is defined as
\begin{equation}\label{eqn:dynamic}
R_{T}^{d} := \frac{1}{n} \sum_{i=1}^{n} \sum_{t=0}^{T} f_{t}(x_{i,t}) - \sum_{t=1}^{T} \min_{x_{t}}  f_{t}(x_{t}).
\end{equation}

\noindent Unlike the static regret $R_T^s$ that can be shown to grow sublinearly with $T$, the dynamic regret $R_T^d$ depends on problem regularities.
The most common regularity is the optimizer path-length, which is the sum of the distance between the individual optimal points at each time, i.e.,
\begin{equation} \label{eqn:PathLength}
\mathcal{P}_{T}:= \sum_{t=0}^{T-1} \|{x_{t+1}^{\ast}-x_{t}^\ast}\|,
\end{equation}
\noindent where the norm is in the $L_{2}$-sense. The distributed online algorithms in \cite{shahrampour2018distributed, nazari2019dadam} both have regret that is asymptotically $O(\sqrt{T} \mathcal{P}_T)$.
In this paper, we will also need what we call the gradient path-length. Let $g_t(x_t) = [\nabla f_{1, t}(x_{1,t}), \dots, \nabla f_{n, t}(x_{n,t})]^T$. Then, the gradient path-length is given by
\begin{equation} \label{eqn:GradPathLength}
\mathcal{V}_{T}:= \sum_{t=0}^{T-1} \|{g_{t+1} - g_{t}}\|_{\infty},
\end{equation}
\noindent where $\| \cdot \|_{\infty}$ denotes the $L_{\infty}$ norm. A similar squared version of $\mathcal{V}_{T}$ known as the gradient variation has appeared in regret bounds before in \cite{chiang2012online, jadbabaie2015online}.
We also make the following assumptions on the local objective function $f_{i,t}$ at each time.

\begin{assumption}
	\label{assum:lipschitz_f}
	For all $i$ and $t$, the function $f_{i,t}$ is $L_f-$Lipschitz, that is, there exists a constant $L_{f} > 0$ such that
	\begin{equation*}
		| f_{i,t} (x) - f_{i,t} (y) | \leq L_f \|x - y\|, \text{ for all } x, y.
	\end{equation*}
\end{assumption}

\begin{assumption}
	\label{assum:LSmooth}
	For all $i$ and $t,$ the function $f_{i,t}$ is $L_{g}$-smooth, i.e. there exists a constant $L_{g} > 0$ such that 
	\begin{equation*}
		\| \nabla f_{i,t} (x) - \nabla  f_{i,t} (y) \| \leq L_g \|x - y\|, \text{ for all } x, y.
	\end{equation*}
\end{assumption}
As previously discussed, we will also make use of the additional assumption of strong convexity on $f_{t}(x)$.
\begin{assumption}
	\label{assum:StrongConvex}
	For all $t,$ the global objective function $f_{t}$ is $\mu$-strongly convex.
\end{assumption}

\noindent Strong convexity is an assumption used in offline convex optimization to prove linear convergence of gradient descent. This is expressed formally in the following result shown in \cite{qu2018harnessing}, which we state without proof.
\begin{lemma}
	\label{lemma:GradConv}
	Let $F(x)$ be $\mu$-strongly convex and its gradient be $\beta$-Lipschitz continuous. Moreover, let $x_{0} \in \mathbb{R}^{d}$ and $0 < \eta < \frac{2}{\beta}$, and define $x_{1} := x_{0} - \eta \nabla F(x).$ Then, for $\lambda:= \max (|1- \eta \mu|,|1-\eta \beta|) < 1,$ we have $\| x_{1}-x^{\ast} \| \leq \lambda \| x_{0} - x^{\ast} \|$.
\end{lemma}

Furthermore, we introduce notations and an assumption on the network structure. Denote the graph as $\mathcal{G} = (\mathcal{V}, \mathcal{E})$, where $\mathcal{V}$ is the agent index set $\{1, \dots, n\}$ and $\mathcal{E}$ is the edge set. $(i, j) \in \mathcal{E}$ if agent $i$ can receive information from agent $j$. We define a weight matrix $W$ that encodes the topology of graph $\mathcal{G}$. Specifically, if $W_{ij} \neq 0$, edge $(i, j) \in \mathcal{E}$ and vice versa. We make the following assumption on graph $\mathcal{G}$ and matrix $W$.
\begin{assumption}
	\label{assum:doubly_stochastic}
	The graph $\mathcal{G}$ is undirected and connected. Furthermore, $W$ is doubly stochastic and symmetric. That is, $W \mathbf{1} = \mathbf{1}$ and $W = W^T$.
\end{assumption}
A direct consequence of Assumption~\ref{assum:doubly_stochastic} is that the $L_2$ matrix norm $\|W - \frac{1}{n}\mathbf{1}\mathbf{1}^T\| = \sigma_W < 1$, where $\sigma_W$ is
also called the mixing rate. When $\sigma_W$ is small, the network reaches consensus faster, see \cite{johansson2008subgradient}.

%% file: tex/convergence.tex
\section{Algorithm and Performance Analysis}
\label{sec:convergence}
\floatstyle{spaceruled}
\restylefloat{algorithm}
\begin{subequations}
	\label{eq:eq_oco}
\begin{algorithm}[t]
	\caption{Distributed Online Optimization with Gradient Tracking}
	\label{alg:d-oco}
	\begin{algorithmic}[1]
		\Require{The primal variable $x_{i,0}$, the local gradient $\nabla f_{i,0}(x_{i,0})$ and global gradient estimate $y_{i,0} = \nabla f_{i,0}(x_{i,0})$ for all $i$, and the doubly stochastic matrix $W := (w_{i,j}).$}
		\For{$1 \leq t \leq T$}
		\State{ For all $i = 1, \dots, n$, agent $i$ computes
			\begin{equation}
				\label{eq:eq_oco_x}
				x_{i,t+1} = \sum_{j \in N_i} W_{ij} (x_{j,t}- \alpha_j y_{j,t}).
			\end{equation}
		}
		\State{ For all $i = 1, \dots, n$, agent $i$ computes
			\begin{equation}
				\label{eq:eq_oco_y}
				\begin{split}
				y_{i,t+1} = \sum_{j \in N_i} W_{ij} y_{j,t} & + \nabla f_{i, t+1}(x_{i,t+1})  \\
				& - \nabla f_{i, t}(x_{i,t}).
				\end{split}
		\end{equation}}
		\EndFor
	\end{algorithmic}
\end{algorithm}
\end{subequations}

Our proposed algorithm is detailed in Algorithm \ref{alg:d-oco}. Without loss of generality, we suppose that the iterates in Algorithm~\ref{alg:d-oco} satisfy $x_{i,t}, y_{i,t} \in \mathbb{R}$. Let $\alpha = \mathtt{diag}(\alpha_1, \dots, \alpha_n)$. The update in \eqref{eq:eq_oco} can be written in a compact form
\begin{subequations}
	\label{eq:x_y_update}
	\begin{equation}
		\label{eq:x_update}
		x_{t+1} = W(x_t - \alpha y_t),
	\end{equation}
	\begin{equation}
\label{eq:y_update}
		y_{t+1} = Wy_t + g_{t+1}(x_{t+1}) - g_t(x_t).
\end{equation}	
\end{subequations}

\noindent The updates \eqref{eq:x_update} and \eqref{eq:y_update} are an online analogue to the algorithm proposed in \cite{pu2018push}. We start by assuming that the agents have access to the local gradient at the present. The algorithm proceeds by alternating between descent and gradient computation steps, simultaneously communicating with other agents. Rather than performing standard distributed gradient descent, we employ a correction that corrects for the difference in the gradient between the current points and those in the immediate past. This technique is shown in \cite{pu2018push} to improve the convergence rates for the offline problem. Denote by $\bar{x}_t = \frac{1}{n} \sum_i x_{i, t}$, $\bar{y}_t = \frac{1}{n} \sum_i y_{i, t}$ the average values of the primal variables and the gradients across the network. We have the following lemma.
\begin{lemma}
	\label{lem:regret_bound}
	Let Assumption~\ref{assum:lipschitz_f} hold. Then, the regret $R_T^d$ is upper bounded by
	\begin{equation*}
		R_T^d \leq \sqrt{n}L_f \sum_{t=0}^T \|x_t - \mathbf{1}\bar{x}_t\| + nL_f	\sum_{t=0}^{T} \|\bar{x}_t - x_t^\star\|.
	\end{equation*}
\end{lemma}
\begin{proof}
	Given Assumption~\ref{assum:lipschitz_f}, for all $i$ and $t$, we have that $ f_t(x_{i,t}) - f_t(x_t^\star) \leq nL_f \|x_{i,t} - x_t^\star\| $. Summing both sides over $i$ and $t$ and dividing by $n$, we have that
	\begin{equation}
		\label{eq:bd_regret}
		R_T^d \leq L_f \sum_{t = 0}^{T} \sum_{i=1}^n \|x_{i,t} - x_t^\star\|.
	\end{equation}
	Adding and subtracting $\bar{x}_t$ inside the norm in \eqref{eq:bd_regret} and applying the triangle inequality, we have that
	\begin{equation*}
		\label{eq:bd_regret_1}
		\begin{split}
		R_T^d & \leq L_f \sum_{t=0}^T \sum_{i=1}^{n} \|x_{i,t} - \bar{x}_t\| + nL_f	\sum_{t=0}^{T} \|\bar{x}_t - x_t^\star\| \\
		& \leq \sqrt{n}L_f \sum_{t=0}^T \|x_t - \mathbf{1}\bar{x}_t\| + nL_f	\sum_{t=0}^{T} \|\bar{x}_t - x_t^\star\|.
		\end{split}
	\end{equation*}
	The second inequality is due to the inequality of the root mean square and the arithmetic mean. The proof is complete. \quad
\end{proof}
From Lemma~\ref{lem:regret_bound}, we observe that the regret is upper bounded by the network error $\sum_{t=0}^T \|x_t - \mathbf{1}\bar{x}_t\|$ and the tracking error $\sum_{t=0}^{T} \|\bar{x}_t - x_t^\star\|$. Next, we shall show that under an appropriate step size, both errors are bounded by the regularity of the problem up to a scaling factor. Before presenting this result, we need the following lemma
\begin{lemma}
	\label{lem:conserve_y}
	Let Assumption~\ref{assum:doubly_stochastic} hold. Then, for all $t$, we have $\mathbf{1}^T y_t = \mathbf{1}^T g_t(x_t)$.
\end{lemma}
\begin{proof}
	According to the initialization of Algorithm~\ref{alg:d-oco}, $y_{i,0} = \nabla f_{i,0}(x_{i,0})$, the above equality is satisfied when $t = 0$. Next, we use mathematical induction: Assuming that $\mathbf{1}^T y_t = \mathbf{1}^T g_t(x_t)$, we need to show $\mathbf{1}^T y_{t+1} = \mathbf{1}^T g_{t+1}(x_{t+1})$. Recalling the update in \eqref{eq:y_update}, we have that $ \mathbf{1}^T y_{t+1} =  \mathbf{1}^T Wy_t +  \mathbf{1}^T g_{t+1}(x_{t+1}) -  \mathbf{1}^T g_t(x_t) =  \mathbf{1}^T y_t +  \mathbf{1}^T g_{t+1}(x_{t+1}) -  \mathbf{1}^T g_t(x_t) = \mathbf{1}^T g_{t+1}(x_{t+1})$. The proof is complete.
\end{proof}

\begin{lemma}
	\label{lem:bound_by_regularity}
	Let Assumptions \ref{assum:LSmooth}, \ref{assum:StrongConvex} and \ref{assum:doubly_stochastic} hold. Then, there exist small enough step sizes $[\alpha_1, \dots, \alpha_n]$ such that $\sum_{t=0}^T\|\bar{x}_t -  x_t^\star\|$, $\sum_{t=0}^T \|x_t - \mathbf{1}\bar{x}_t\|$ and $\sum_{t=0}^T \|y_t - \mathbf{1}\bar{y}_t\|$ are all of order $O(C_1 + C_2 + C_3 + \mathcal{P}_T + \mathcal{V}_T)$, where $C_1 = \|\bar{x}_0 - x_0^\star\|$, $C_2 = \|x_0 - \mathbf{1}\bar{x}_0\|$ and $C_3 = \|y_0 - \mathbf{1}\bar{y}_0\|$. Furthermore, if $\alpha_1=...=\alpha_n$, then the step sizes can be chosen as $\alpha_i < \frac{1-\sigma_W^2}{1 - \sigma_W^2 + 2\sigma_W} \frac{1}{1 + n \frac{L_g}{\mu}} \frac{1}{L_g}$, for all $i$.
\end{lemma}
\begin{proof}
	The first step is to show the contraction of $\|\bar{x}_t - x_t^\star\|$,  $\|x_t - \mathbf{1}\bar{x}_t\|$ and $\|y_t - \mathbf{1}\bar{y}_t\|$ over time. According to \eqref{eq:x_update}, we have that
	\begin{equation}
	\label{eq:eq_1}
		\begin{split}
		& \|\bar{x}_{t+1} - x_{t+1}^\star\| = \|\frac{1}{n}\mathbf{1}^Tx_{t+1} - x_{t+1}^\star\| \\
		&= \|\frac{1}{n}\mathbf{1}^TW(x_t - \alpha y_t) - x_t^\star + x_t^\star - x_{t+1}^\star\| \\
		&\leq \|\bar{x}_t - \frac{1}{n}\mathbf{1}^T\alpha y_t - x_{t}^\star\| + \|x_t^\star - x_{t+1}^\star\|.
		\end{split}
	\end{equation}
	Denote $\bar{y}_t = \frac{1}{n}\mathbf{1}^T y_t$, $\bar{g}_t  = \frac{1}{n}\mathbf{1}^T g_t(\mathbf{1}\bar{x}_t)$ and $\alpha' = \frac{1}{n}\mathbf{1}^T\alpha \mathbf{1}$. Adding and subtracting $\alpha' \bar{g}_t$ and $\alpha' \bar{y}_t$ in $\|\bar{x}_t - \frac{1}{n}\mathbf{1}^T\alpha y_t - x_{t}^\star\|$ in \eqref{eq:eq_1} and using the triangle inequality, we have that
	\begin{equation}
		\label{eq:eq_2}
		\begin{split}
		\| \bar{x}_t - \frac{1}{n}\mathbf{1}^T & \alpha y_t - x_{t}^\star\| \leq  \|\bar{x}_t - \alpha' \bar{g}_t - x_{t}^\star \|  \\
		& +  \alpha'\| \bar{g}_t -  \bar{y}_t\|  +  \frac{\|\mathbf{1}^T\alpha\|}{n}\|y_t - \mathbf{1} \bar{y}_t\|.
		\end{split}
	\end{equation}
	According to Lemma 10 in \cite{qu2018harnessing} and the $\mu-$strong convexity of $f_t(x)$, we have that $ \|\bar{x}_t - \alpha' \bar{g}_t - x_{t}^\star \| \leq (1 - \frac{1}{n}\alpha' \mu) \|\bar{x}_t - x_{t}^\star\|$ if the step size satisfies $\frac{1}{n}\alpha' \in (0, \frac{1}{n L_g})$. In addition, due to Lemma~\ref{lem:conserve_y} and Assumption~\ref{assum:LSmooth}, we have that $\|\bar{g}_t -  \bar{y}_t\| = \frac{1}{n} \|\mathbf{1}^T g_t(\mathbf{1}\bar{x}_t) - \mathbf{1}^T g_t(x_t)\| \leq \frac{L_g}{\sqrt{n}} \|x_t - \mathbf{1}\bar{x}_t\|$. Combining these inequalities with \eqref{eq:eq_1} and \eqref{eq:eq_2}, we have that
	$\|\bar{x}_{t+1} - x_{t+1}^\star\| \leq (1 - \frac{1}{n} \alpha' \mu)\|\bar{x}_t - x_t^\star\| + \alpha' \frac{L_g}{\sqrt{n}}\| x_t - \mathbf{1}\bar{x}_t\| + \frac{\|\mathbf{1}^T\alpha\|}{n}\|y_t - \mathbf{1} \bar{y}_t\| + \|x_t^\star - x_{t+1}^\star\|. $
	Adding both sides of the above inequality from $t = 0$ to $t = T-1$, adding $\|\bar{x}_0 - x_0^\star\|$ to both sides and adding $(1 - \frac{1}{n}\alpha' \mu) \|\bar{x}_T - x_T^\star\|$ on the right hand side, and rearranging the terms, we have that
	\begin{equation}
		\label{eq:eq_3}
		\begin{split}
		& \sum_{t=0}^T\|\bar{x}_t -  x_t^\star\| \leq \frac{C_1}{\beta_1} + \frac{\sqrt{n}L_g}{\mu} 
\sum_{t=0}^T \| x_t - \mathbf{1}\bar{x}_t\| \\ 
& + \frac{\|\mathbf{1}^T\alpha\|}{\alpha' \mu} \sum_{t=0}^T \|y_t - \mathbf{1} \bar{y}_t\| + \frac{n}{\alpha' \mu} \sum_{t=1}^{T} \|x_t^\star - x_{t-1}^\star\|,
		\end{split}
	\end{equation}
	where $C_1 = \|\bar{x}_0 - x_0^\star\|$, $\beta_1 = \frac{1}{n}\alpha' \mu$.
	Next, we bound the term $\sum_{t=0}^T \|x_t - \mathbf{1}\bar{x}_t\|$. According to \eqref{eq:x_update}, we have that
	\begin{equation}
		\label{eq:eq_4}
		\begin{split}
		& \|x_{t+1} - \mathbf{1}\bar{x}_{t+1}\| = \|W(x_t - \alpha y_t) - \frac{1}{n}\mathbf{1}\mathbf{1}^TW(x_t - \alpha y_t)\| \\
		& \leq \|W - \frac{\mathbf{1}\mathbf{1}^T}{n}\|\|x_t - \mathbf{1}\bar{x}_t\| + \|W - \frac{\mathbf{1}\mathbf{1}^T}{n}\|\|\alpha\|\|y_t\|.
		\end{split}
	\end{equation}
	Furthermore, we get that
	\begin{flalign}
	\label{eq:eq_5}
	& \|y_t\| \leq \|y_t - \mathbf{1}\bar{y}_t\| + \|\mathbf{1}\|\|\bar{y}_t\| & \nonumber \\
	& \leq \|y_t - \mathbf{1}\bar{y}_t\| + \|\mathbf{1}\|\|\bar{y}_t - \bar{g}_t\| + \|\mathbf{1}\| \|\bar{g}_t - \frac{1}{n}\mathbf{1}^Tg_t(\mathbf{1}x_t^\star)\| & \nonumber \\
	& \leq \|y_t - \mathbf{1}\bar{y}_t\| + L_g \|x_t - \mathbf{1}\bar{x}_t\| + \sqrt{n} L_g \|\bar{x}_t - x_t^\star\|. &
	\end{flalign}
	The second inequality above is due to the fact that $x_t^\star$ is the optimizer at time $t$ and $\mathbf{1}^Tg_t(\mathbf{1}x_t^\star) = 0$. The last inequality is due to the same reasoning following \eqref{eq:eq_2}.	Recalling that $\|W - \frac{\mathbf{1}\mathbf{1}^T}{n} \| = \sigma_W < 1$, and combining the inequalities \eqref{eq:eq_4} and \eqref{eq:eq_5}, we have that 
	\begin{equation}
		\label{eq:eq_6}
		\begin{split}
		& \|x_{t+1} - \mathbf{1}\bar{x}_{t+1}\| \leq \sigma_W(1 + L_g\|\alpha\|)\|x_t - \mathbf{1}\bar{x}_t\| \\
		&\quad  + \sigma_W\|\alpha\|\|y_t - \mathbf{1}\bar{y}_t\| + \sigma_W \sqrt{n}L_g \|\alpha\| \|\bar{x}_t - x_t^\star\|.
		\end{split}
	\end{equation}
	Next, we bound the network error $\sum_{t=0}^T \|y_t - \mathbf{1}\bar{y}_t\|$. According to \eqref{eq:y_update}, we have that 
	$ \|y_{t+1} - \mathbf{1} \bar{y}_{t+1}\| = \|(W - \frac{1}{n}\mathbf{1}\mathbf{1}^T) (y_t - \mathbf{1} \bar{y}_t) + (I - \frac{1}{n}\mathbf{1}\mathbf{1}^T)(g_{t+1}(x_{t+1}) - g_t(x_t))\| \leq \|W - \frac{1}{n}\mathbf{1}\mathbf{1}^T\| \|y_t - \mathbf{1} \bar{y}_t\| + \|I - \frac{1}{n}\mathbf{1}\mathbf{1}^T\| \|g_{t+1}(x_{t+1}) - g_t(x_t)\|$, 	
	where the last inequality is due to the triangle and the Cauchy-Schwartz inequalities. 	
	Adding and subtracting $g_{t+1}(x_t)$ in the norm $\|g_{t+1}(x_{t+1}) - g_t(x_t)\|$ and using the triangle inequality, we have that
	\begin{flalign}
\label{eq:eq_y_2}
& \|y_{t+1} - \mathbf{1} \bar{y}_{t+1}\| \leq \sigma_W \|y_t - \mathbf{1} \bar{y}_t\| + \|g_{t+1}(x_{t+1})  & \nonumber \\
& - g_{t+1}(x_t)\| + \| g_{t+1}(x_t) - g_t(x_t)\| \leq \sigma_W \|y_t - \mathbf{1} \bar{y}_t\| & \nonumber \\
& + L_g \|x_{t+1} - x_t\| + \|g_{t+1}(x_t) - g_t(x_t)\|. &
\end{flalign}
	Accoring to \eqref{eq:x_update}, we have that $\|x_{t+1} - x_t\| = \|W(x_t - \alpha y_t) - x_t\| = \|(W - I)x_t - W \alpha y_t\| \leq \|W- I\| \|x_t - \mathbf{1}\bar{x}_t\| + \|W\|\|\alpha\|\|y_t\|$. Combining this inequality with \eqref{eq:eq_5} and \eqref{eq:eq_y_2}, and recalling that $\|W\| = 1$, we have that
	\begin{flalign}
	\label{eq:eq_7}
	& \|y_{t+1} - \mathbf{1} \bar{y}_{t+1}\| \leq (\sigma_W + L_g\|\alpha\|)\|y_{t} - \mathbf{1} \bar{y}_{t}\| & \nonumber \\
	& \quad \quad \;\; + L_g(\|W - I\| + L_g\|\alpha\|)\|x_t - \mathbf{1}\bar{x}_t\| &  \\
	& \quad \quad \;\; + \sqrt{n}L_g^2\|\alpha\|\|\bar{x}_t - x_t^\star\| + \|g_{t+1}(x_{t}) - g_t(x_t)\|. & \nonumber
	\end{flalign}
	
	Next, we telescope the inequalities \eqref{eq:eq_6} and \eqref{eq:eq_7} and rearrange terms as how we get \eqref{eq:eq_3}. We can obtain
	\begin{equation}
		\label{eq:eq_8}
		\begin{split}
		\sum_{t=0}^T \|x_t - \mathbf{1}\bar{x}_t\| & \leq \frac{C_2}{\beta_2} + \frac{\sigma_W \|\alpha\|}{\beta_2} \sum_{t=0}^T\|y_t - \mathbf{1}\bar{y}_t\| \\
		& + \frac{\sigma_W \sqrt{n} L_g \|\alpha\|}{\beta_2} \sum_{t=0}^{T} \|\bar{x}_t - x_t^\star\|,
		\end{split}
	\end{equation}
	where $C_2 = \|x_0 - \mathbf{1}\bar{x}_0\|$, $\beta_2 = 1 - \sigma_W(1 + L_g\|\alpha\|)$, and
	\begin{flalign}
	\label{eq:eq_9}
	& \sum_{t=0}^T \|y_t - \mathbf{1}\bar{y}_t\| \leq \frac{C_3}{\beta_3} + \frac{L_g(\|W - I\| + L_g\|\alpha\|)}{\beta_3} & \nonumber \\
	& \quad \quad \quad \sum_{t=0}^T \|x_t - \mathbf{1}\bar{x}_t\|  + \frac{\sqrt{n}L_g^2\|\alpha\|}{\beta_3} \sum_{t=0}^T \|\bar{x}_t - x_t^\star\| & \\
	& \quad \quad \quad + \frac{1}{\beta_3} \sum_{t=1}^{T} \|g_{t}(x_{t-1}) - g_{t-1}(x_{t-1})\|, & \nonumber
	\end{flalign}
	where $C_3 = \|y_0 - \mathbf{1}\bar{y}_0\|$, $\beta_3 = 1 - \sigma_W - L_g\|\alpha\|$. 
	
	Next, we derive the upper bound on the network errors $\sum_{t=0}^T \|x_t - \mathbf{1}\bar{x}_t\|$ and $\sum_{t=0}^T \|y_t - \mathbf{1}\bar{y}_t\|$ using inequalities \eqref{eq:eq_3}, \eqref{eq:eq_8} and \eqref{eq:eq_9}. To achieve this, we scale \eqref{eq:eq_8} and \eqref{eq:eq_9} by positive factors $M$ and $N$ and add up the three inequalities. After rearranging terms, the desired upper bound in Lemma~\ref{lem:bound_by_regularity} is derived if the following conditions are satisfied: the term $\sum_{t=0}^T\|\bar{x}_t -  x_t^\star\|$ on both sides is cancelled, the terms $\sum_{t=0}^T \|x_t - \mathbf{1}\bar{x}_t\|$ and $\sum_{t=0}^T \|y_t - \mathbf{1}\bar{y}_t\|$ appear on the left hand side and the regularity terms $\sum_{t=1}^{T} \|x_t^\star - x_{t-1}^\star\|$ and $\sum_{t=1}^{T} \|g_{t}(x_{t-1}) - g_{t-1}(x_{t-1})\|$ appear on the right hand side. In the following, we show that there always exists step size $\alpha$ so that positive scaling factors $M$ and $N$ can be found to satisfy these conditions. To do so, we let $M$ and $N$ satisfy the following conditions:
	\begin{subequations}
		\label{eq:eq_10}
		\begin{flalign}
		\label{eq:eq_10_1}
		& M \frac{\sigma_W \sqrt{n} L_g \|\alpha\|}{\beta_2} + N \frac{\sqrt{n}L_g^2 \|\alpha\|}{\beta_3} = 1,
		\end{flalign}
		\begin{flalign}
		\label{eq:eq_10_2}
		& M > \frac{\sqrt{n}L_g}{\mu} + N  \frac{L_g(\|W - I\| + L_g\|\alpha\|)}{\beta_3},
		\end{flalign}
		\begin{flalign}
		\label{eq:eq_10_3}
		& N > \frac{\|\mathbf{1}^T\alpha\|}{\alpha'\mu} + M \frac{\sigma_W \|\alpha\|}{\beta_2}.
		\end{flalign}
	\end{subequations} 	
		Moreover, we define $M = \frac{\beta_2}{\sigma_W \sqrt{n} L_g \|\alpha\|} b$ and $N = \frac{\beta_3}{\sqrt{n}L_g^2\|\alpha\|} (1-b)$, where $b \in (0, 1)$. Then, the condition~\eqref{eq:eq_10_1} is automatically satisfied. Plugging the expressions of $M$ and $N$ into inequalities \eqref{eq:eq_10_2} and \eqref{eq:eq_10_3}, we have that
		\begin{subequations}
			\label{eq:eq_11}
		\begin{equation}
			\label{eq:eq_11_1}
			b < \frac{\beta_3 - \sqrt{n}L_g^2\frac{\|\mathbf{1}^T\alpha\|}{\alpha'\mu} \|\alpha\|}{\beta_3 + L_g\|\alpha\|} < 1,
		\end{equation}
		\begin{equation}
			\label{eq:eq_11_2}
			b > \frac{(n\frac{L_g^2}{\mu}+L_g)\|\alpha\| + \|W - I\|}{L_g\|\alpha\| + \|W - I\| + \frac{\beta_2}{\sigma_W}} > 0.
		\end{equation}
		\end{subequations}		
		
		As long as the interval of $b$ defined in \eqref{eq:eq_11} is nonempty, we can use any $b$ in this interval to find the corresponding $M$ and $N$ that satisfy the conditions in \eqref{eq:eq_10}. When $\|\alpha\|$ goes to zero, the upper bound of $b$ in \eqref{eq:eq_11_1} goes to $1$ and the lower bound of $b$ in \eqref{eq:eq_11_2} goes to $\frac{\|W - I\|}{\|W - I\| + \frac{1 - \sigma_W}{\sigma_W}} < 1$. Therefore, there always exists a step size $\alpha$ small enough so that $b$, $M$ and $N$ exist to derive the desired bound on $\sum_{t=0}^T \|x_t - \mathbf{1}\bar{x}_t\|$ and $\sum_{t=0}^T \|y_t - \mathbf{1}\bar{y}_t\|$ in Lemma~\ref{lem:bound_by_regularity}.
		According to the inequality \eqref{eq:eq_3}, using the bound on $\sum_{t=0}^T \|x_t - \mathbf{1}\bar{x}_t\|$ and $\sum_{t=0}^T \|y_t - \mathbf{1}\bar{y}_t\|$, it is simple to see that $\sum_{t=0}^T\|\bar{x}_t -  x_t^\star\|$ is also bounded by the same order.
		
		In the remainder of this proof, we show that if $\alpha_1 = \alpha_2 = \dots = \alpha_n$ and $\alpha_i < \frac{1-\sigma_W^2}{1 - \sigma_W^2 + 2\sigma_W} \frac{1}{1 + n \frac{L_g}{\mu}} \frac{1}{L_g}$ for all $i$, then $\alpha$ satisfies the sufficient decrease requirement in Lemma~\ref{lemma:GradConv}, the contraction requirement in \eqref{eq:eq_6} and \eqref{eq:eq_7}, that is, $\beta_2, \beta_3 > 0$, and the requirement that the interval of $b$ is nonempty in \eqref{eq:eq_11}. To see this, let $\alpha_1 = \dots = \alpha_n = a$. First, we need to choose $a$ such that $\beta_2, \beta_3 > 0$. Plugging $a$ into the expression of $\beta_2$ and $\beta_3$, we have that
		\begin{subequations}
			\label{eq:step}
		$a < \frac{1}{\sigma_W} \frac{1 - \sigma_W}{L_g}$ and $a < \frac{1 - \sigma_W}{L_g}$.
		Since $\sigma_W < 1$, it is sufficent to choose
		\begin{equation}
		\label{eq:step_1}
		a < \frac{1 - \sigma_W}{L_g}
		\end{equation}
		to obtain $\beta_2, \beta_3 > 0$.	
		Moreover, we need to choose $a$ so that the upper bound in \eqref{eq:eq_11_1} is strictly greater than \eqref{eq:eq_11_2}. Because $\alpha_1 = \dots = \alpha_n = a$, we have that $\|\alpha\| = a$ and $\alpha' = \frac{1}{n}\mathbf{1}^T\alpha\mathbf{1} = a$. Plugging these into the bounds in \eqref{eq:eq_11}, we obtain that
		\begin{equation*}
			\frac{(\frac{nL_g^2}{\mu} + L_g)a + \|W - I\|}{L_ga + \frac{\beta_2}{\sigma_W} + \|W - I\|} < \frac{\beta_3 - n \frac{L_g^2}{\mu} a}{\beta_3 + L_ga}.
		\end{equation*}
		By Assumption~\ref{assum:doubly_stochastic}, we have that all eigenvalues of $W$ satisfy $|\lambda(W)| \leq 1$ due to Perron-Frobenius theory. Therefore, $\|W - I\| \leq 2$. Therefore, to satisfy the above inequality, it is sufficient to let 
		\begin{equation*}
			\frac{(\frac{nL_g^2}{\mu} + L_g)a + 2}{L_ga + \frac{\beta_2}{\sigma_W} + 2} < \frac{\beta_3 - n \frac{L_g^2}{\mu} a}{\beta_3 + L_ga} \text{ and} \;\;\; \beta_3 - n \frac{L_g^2}{\mu} a > 0,
		\end{equation*}
		due to the inequality $\frac{p+m}{q+m} > \frac{p}{q}$ for any $p > q > 0$ and $m > 0$. Simplifying the above two inequalities, we obtain the following conditions on $a$
		\begin{equation}
			\label{eq:step_2}
			a < \frac{1-\sigma_W^2}{1 - \sigma_W^2 + 2\sigma_W} \frac{1}{1 + n \frac{L_g}{\mu}} \frac{1}{L_g},
			\end{equation}
			\begin{equation}
			\label{eq:step_3}
			a < (1-\sigma_W) \frac{1}{1 + n \frac{L_g}{\mu}} \frac{1}{L_g}.
			\end{equation}
			In addition, recalling the analysis after \eqref{eq:eq_2}, we have that
			\begin{equation}
			\label{eq:step_4}
			a < \frac{1}{L_g}.
			\end{equation}
		\end{subequations}
		It is simple to see that all the step size conditions in \eqref{eq:step} reduce to the condition \eqref{eq:step_2}.
		The proof is complete.
\end{proof}

\begin{remark}
	Compared to existing methods in distributed online optimization, e.g., \cite{shahrampour2018distributed}, here we remove the dependency of the step size on the horizon $T$. 
	According to the upper bound on the stepsize in Lemma~\ref{lem:bound_by_regularity}, when the network size $n$ becomes larger, or the mixing rate $\sigma_W$ is closer to $1$, (that is, the mixing process is slower,) or $\frac{nL_g}{\mu}$ becomes larger, (that is, the central problem is more ill-conditioned), the step size needs to be more conservative.
\end{remark}


\begin{theorem}
	\label{thm:regret}
		\textnormal{
		Let Assumptions~\ref{assum:lipschitz_f}-\ref{assum:doubly_stochastic} hold. Then, there exist small enough stepsizes $[\alpha_1, \dots, \alpha_n]$ such that the regret $R_T^d$ is of the order $O(C_1 + C_2 + C_3 + \mathcal{P}_T + \mathcal{V}_T)$, where $C_1 = \|\bar{x}_0 - x_0^\star\|$, $C_2 = \|x_0 - \mathbf{1}\bar{x}_0\|$ and $C_3 = \|y_0 - \mathbf{1}\bar{y}_0\|$. Furthermore, if $\alpha_1=...=\alpha_n$, then the step sizes can be chosen as  $\alpha_i < \frac{1-\sigma_W^2}{1 - \sigma_W^2 + 2\sigma_W} \frac{1}{1 + n \frac{L_g}{\mu}} \frac{1}{L_g}$, for all $i$.}
\end{theorem}

\begin{proof}
	This theorem is a direct result of combining Lemma~\ref{lem:regret_bound} and \ref{lem:bound_by_regularity}.
\end{proof}

We conclude this section by remarking on our assumption that $W$ is doubly stochastic. In \cite{pu2018push}, the matrix $W$ in \eqref{eq:eq_oco_x} and \eqref{eq:eq_oco_y} is replaced by a row and column stochastic matrix respectively.
It is easy to see that our analysis can be adapted to reach a similar conclusion, noting that additional terms should be involved in our bounds related to norm equivalence. We do not pursue this further.

%% file: tex/simulation.tex
\section{Numerical Simulation}
\label{sec:simulation}

In this section, we validate our theorectical analysis using a distributed tracking example. Furthermore, we compare Algorithm~\ref{alg:d-oco} to the distributed online gradient descent algorithm in \cite{shahrampour2018distributed} which does not utilize the gradient tracking technique.
\begin{figure}
	\centering
	\subfigure[]{
		\includegraphics[width = .7\columnwidth]{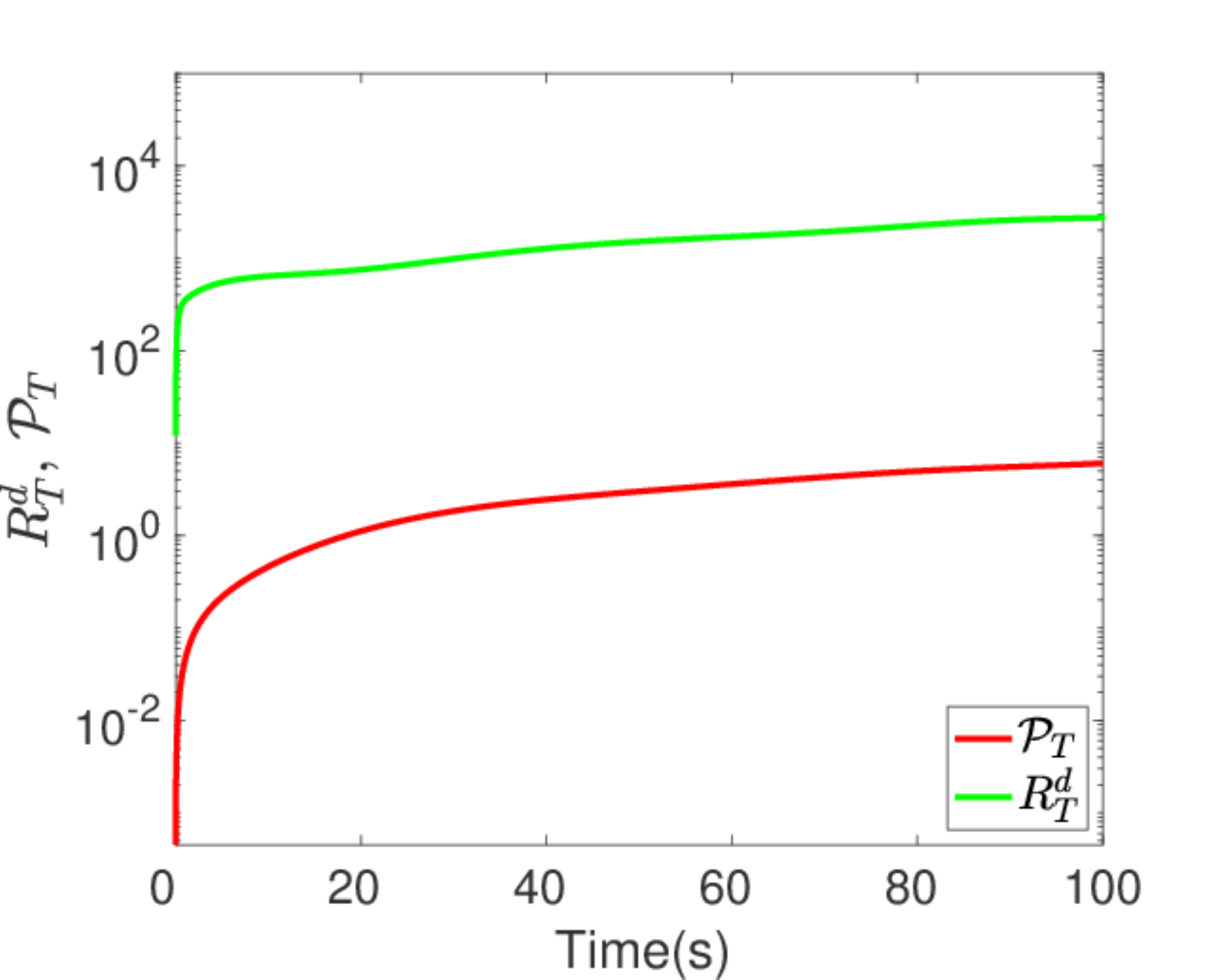}
	}
	\subfigure[]{
		\includegraphics[width = .7\columnwidth]{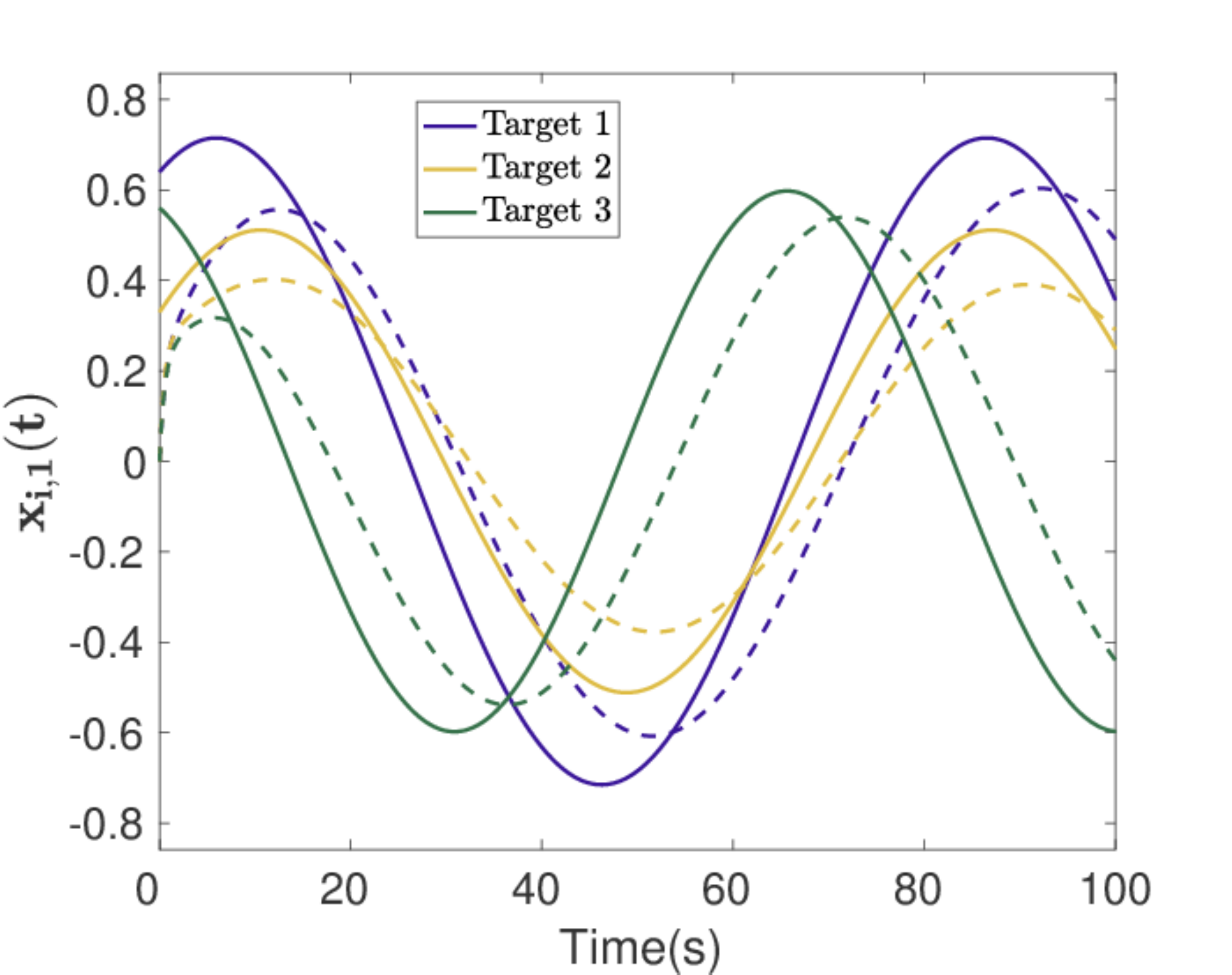}
	}
	\caption{Distributed tracking using sampling interval $0.01s$ and step size $0.003$ for all agents. (a) shows the regret of the solution and the regularity of the optimizer path-length. (b) shows the tracking performance at agent $2$ for three moving targets. The solid curves correspond to the true trajectory of target positions $\{x_{i,1}(t)\}$ over time. The dashed curves correspond to the estimated trajectory solved by Algorithm~\ref{alg:d-oco} at agent 2.}
	\label{fig:theoretical_bound_performance}
	\vspace{-4pt}
\end{figure}
Consider a sensor network of size $n = 10$. These sensors collaborate to track three moving targets. Every target $i$ has a state variable of dimension $2$, i.e. $x_i^\ast(t) = [x_{i,1}^\ast(t), x_{i, 2}^\ast(t)]^T$. The position $x_{i,1}^\ast(t)$ is tracking a sinusoidal curve $x_{i,1}^\ast(t) = A_i \sin(\omega_i t + \phi_i)$ and the velocity satisfies $x_{i,2}^\ast(t) = \dot{x}_{i,1}^\ast(t)$. 
We stack all target states into one vector $x(t)^\ast \in \mathbb{R}^{6}$. Each sensor receives a local observation $y_i(t) \in \mathbb{R}^{4}$, with each entry of $y_i(t)$ being a linear combination of $x(t)^\ast$, $y_{i}(t) = C_i x^\ast(t)$,
where $C_i \in \mathbb{R}^{4 \times 6}$.
To estimate the state variable $x(t)^\ast$, the local optimization problem at each sensor is
\begin{equation*}
	\min_{x} \; f_{i,t}(x) = \frac{1}{2} \|C_ix(t) - y_i(t)\|^2,
\end{equation*}
which is under-determined, and the sensors need to communicate with each other to solve the global optimization problem at each time step
\begin{equation*}
	\min_{x} \; f_t(x) = \frac{1}{2} \|C x(t) - y(t)\|^2,
\end{equation*}
where $C = [C_1^T | \dots | C_{n}^T ]^T$ and $y(t) = [y_1(t)^T, \dots,$ $y_n(t)^T]^T$. When there are enough measurements, the above central problem is strongly convex. 
In this specific example, the gradient-path regularity $\mathcal{V}_T$ in \eqref{eqn:GradPathLength} is of the same order of $\mathcal{P}_T$ in \eqref{eqn:PathLength}. 
To see this, we observe that the norm $\|{\nabla f_{i, t+1}-\nabla f_{i, t}}\|_{\infty} = \|C_i^T(y(t+1) - y(t))\|  \leq \|C_i^TC_i\| \|x(t+1)^\ast - x(t)^\ast\|$. Therefore, in the following, we compare the regret directly with $\mathcal{P}_T$ neglecting $\mathcal{V}_T$.

\begin{figure}[t]
	\centering
	\subfigure[]{
		\includegraphics[width = .7\columnwidth]{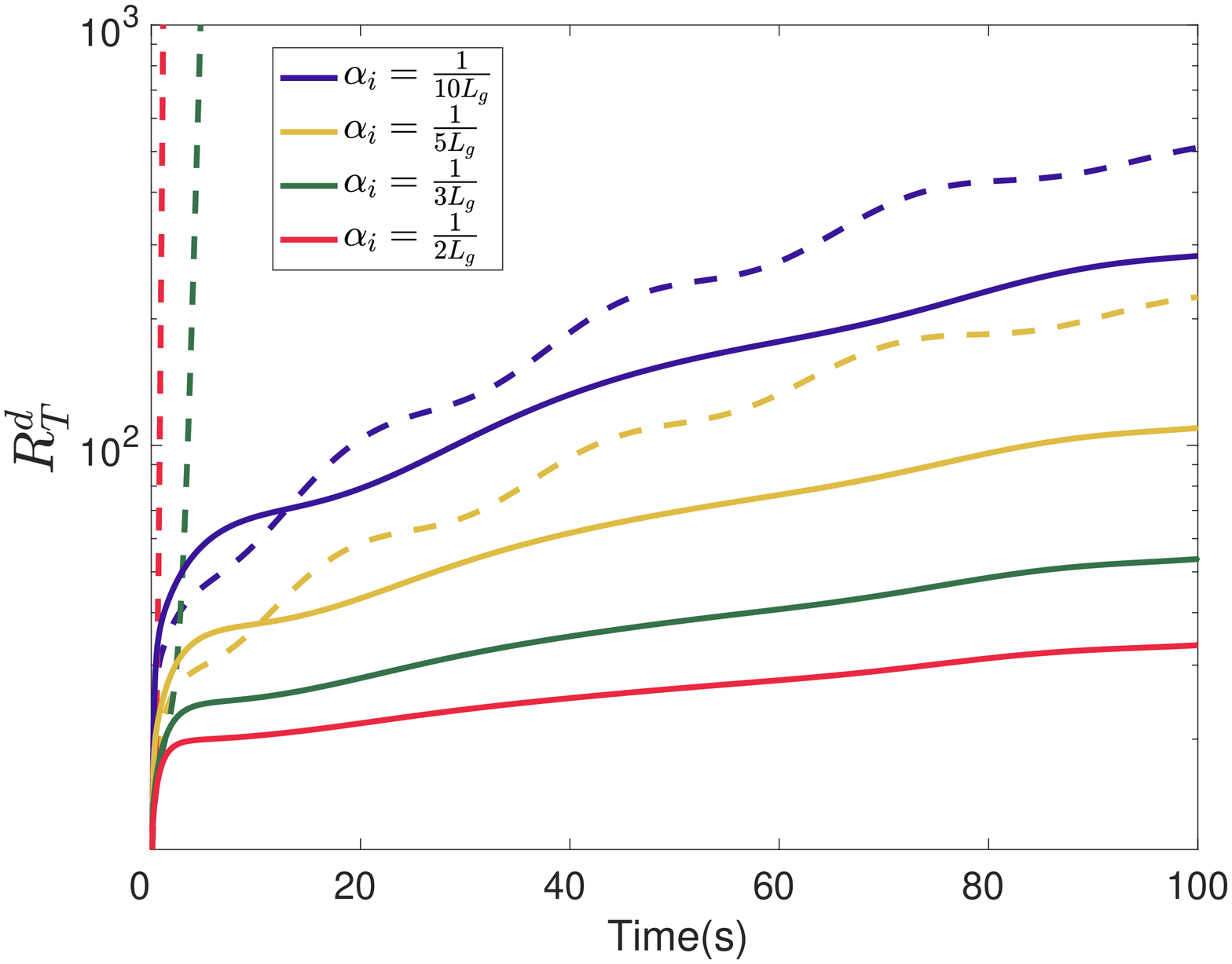}
	}
	\subfigure[]{
		\includegraphics[width = .7\columnwidth]{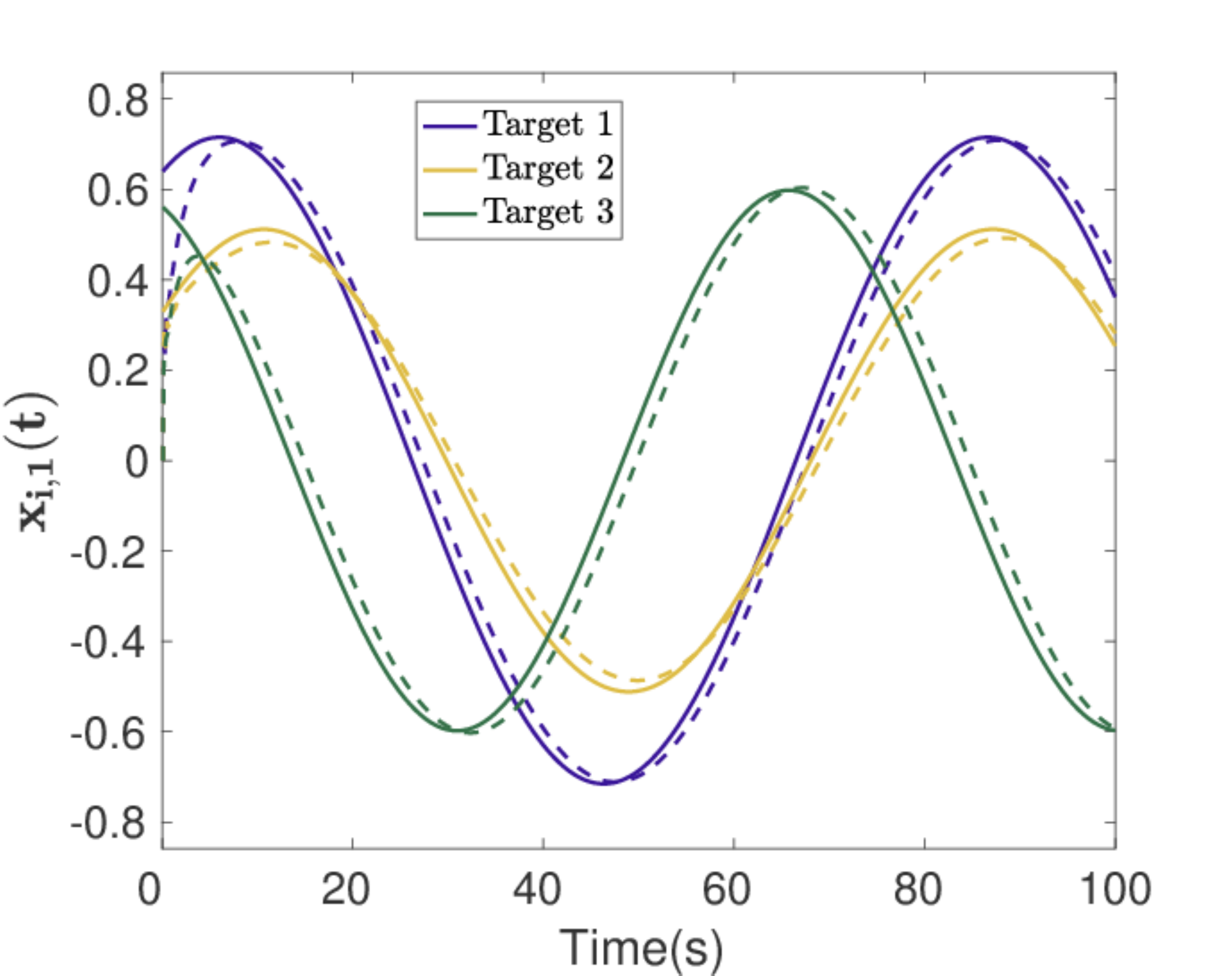}
	}
	\caption{Distributed tracking using sampling interval $0.1s$. (a) shows the comparison between the distributed online optimization with (solid lines) and without gradient tracking (dashed lines) with different step sizes. (b) shows the tracking performance using sampling interval $0.1s$ and the stepsize $\frac{1}{2L_g}$.}
	\label{fig:practical_bound_performance}
	\vspace{-2pt}
\end{figure}


In our simulation, the target parameters $\{A_i, \omega_i, \phi_i\}$ and observation matrices $\{C_i\}$ are randomly generated. 
In addition, the doubly stochastic weight matrix $W$ is randomly generated and its mixing rate is $\sigma_W = 0.59$. 
The stepsizes $\alpha_i$ are selected to be $0.003$ according to the bound given in Theorem~\ref{thm:regret}. Applying Algorithm~\ref{alg:d-oco} with this step size and sampling from the target trajectory at frequency $100$Hz, the regret of solving the online optimization problem and the tracking performance is presented in Figure~\ref{fig:theoretical_bound_performance}. From Figure~\ref{fig:theoretical_bound_performance}(a), we observe that the regret is of the same order as the regularity term $\mathcal{P}_T$, as we have shown in Theorem~\ref{thm:regret}, and the local sensors are tracking the moving targets within their neighborhood.


In what follows, we show that Algorithm~\ref{alg:d-oco} can admit more aggressive stepsizes than the algorithm without gradient tracking \cite{shahrampour2018distributed}. In this way, the achieved regret and tracking performance can be further improved. Specifically, we select the sampling frequency to be $10$Hz, and present the comparison of the achieved regret between our algorithm and the one without gradient tracking in \cite{shahrampour2018distributed} in Figure~\ref{fig:practical_bound_performance}(a). 
We observe that the regret of our algorithm always grows slower than the method in \cite{shahrampour2018distributed} with the same choice of stepsize. Moreover, the method in \cite{shahrampour2018distributed} diverges when the stepsizes are chosen to be $\frac{1}{3L_g}$ or $\frac{1}{2L_g}$ while our algorithm achieves lower regrets.
In Figure~\ref{fig:practical_bound_performance}(b), we demonstrate that such stepsizes can achieve better tracking performance than the one obtained using the proposed theoretical bound.

%% file: tex/conclusion.tex
\section{CONCLUSIONS}
\label{sec:conclusion}
In this paper we proposed a new algorithm for distributed online convex optimization based on recent developments in the offline setting. We showed that under the additional assumption of strong convexity and fixed step size, our algorithm can achieve a bound on the dynamic regret without explicit dependence on the number of time steps, in contrast to existing relevant methods. We also illustrated the performance of our algorithm on a distributed tracking example to validate our theoretical results.

%% file: root.bbl
\begin{thebibliography}{10}
\providecommand{\url}[1]{#1}
\csname url@samestyle\endcsname
\providecommand{\newblock}{\relax}
\providecommand{\bibinfo}[2]{#2}
\providecommand{\BIBentrySTDinterwordspacing}{\spaceskip=0pt\relax}
\providecommand{\BIBentryALTinterwordstretchfactor}{4}
\providecommand{\BIBentryALTinterwordspacing}{\spaceskip=\fontdimen2\font plus
\BIBentryALTinterwordstretchfactor\fontdimen3\font minus
  \fontdimen4\font\relax}
\providecommand{\BIBforeignlanguage}[2]{{%
\expandafter\ifx\csname l@#1\endcsname\relax
\typeout{** WARNING: IEEEtran.bst: No hyphenation pattern has been}%
\typeout{** loaded for the language `#1'. Using the pattern for}%
\typeout{** the default language instead.}%
\else
\language=\csname l@#1\endcsname
\fi
#2}}
\providecommand{\BIBdecl}{\relax}
\BIBdecl

\bibitem{rabbat2004distributed}
M.~Rabbat and R.~Nowak, ``Distributed optimization in sensor networks,'' in
  \emph{Proceedings of the 3rd international symposium on Information
  processing in sensor networks}.\hskip 1em plus 0.5em minus 0.4em\relax ACM,
  2004, pp. 20--27.

\bibitem{ram2009distributed}
S.~S. Ram, V.~V. Veeravalli, and A.~Nedic, ``Distributed non-autonomous power
  control through distributed convex optimization,'' in \emph{IEEE INFOCOM
  2009}.\hskip 1em plus 0.5em minus 0.4em\relax IEEE, 2009, pp. 3001--3005.

\bibitem{shahrampour2016distributed}
S.~Shahrampour, A.~Rakhlin, and A.~Jadbabaie, ``Distributed detection:
  Finite-time analysis and impact of network topology,'' \emph{IEEE
  Transactions on Automatic Control}, vol.~61, no.~11, pp. 3256--3268, 2016.

\bibitem{nedic2017fast}
A.~Nedi{\'c}, A.~Olshevsky, and C.~A. Uribe, ``Fast convergence rates for
  distributed non-bayesian learning,'' \emph{IEEE Transactions on Automatic
  Control}, vol.~62, no.~11, pp. 5538--5553, 2017.

\bibitem{zinkevich2003online}
M.~Zinkevich, ``Online convex programming and generalized infinitesimal
  gradient ascent,'' in \emph{Proceedings of the 20th International Conference
  on Machine Learning (ICML-03)}, 2003, pp. 928--936.

\bibitem{hazan2007logarithmic}
E.~Hazan, A.~Agarwal, and S.~Kale, ``Logarithmic regret algorithms for online
  convex optimization,'' \emph{Machine Learning}, vol.~69, no. 2-3, pp.
  169--192, 2007.

\bibitem{hosseini2013online}
S.~Hosseini, A.~Chapman, and M.~Mesbahi, ``Online distributed optimization via
  dual averaging,'' in \emph{52nd IEEE Conference on Decision and
  Control}.\hskip 1em plus 0.5em minus 0.4em\relax IEEE, 2013, pp. 1484--1489.

\bibitem{akbari2015distributed}
M.~Akbari, B.~Gharesifard, and T.~Linder, ``Distributed online convex
  optimization on time-varying directed graphs,'' \emph{IEEE Transactions on
  Control of Network Systems}, vol.~4, no.~3, pp. 417--428, 2015.

\bibitem{lee2017sublinear}
S.~Lee and M.~M. Zavlanos, ``On the sublinear regret of distributed primal-dual
  algorithms for online constrained optimization,'' \emph{arXiv preprint
  arXiv:1705.11128}, 2017.

\bibitem{hall2015online}
E.~C. Hall and R.~M. Willett, ``Online convex optimization in dynamic
  environments,'' \emph{IEEE Journal of Selected Topics in Signal Processing},
  vol.~9, no.~4, pp. 647--662, 2015.

\bibitem{jadbabaie2015online}
A.~Jadbabaie, A.~Rakhlin, S.~Shahrampour, and K.~Sridharan, ``Online
  optimization: Competing with dynamic comparators,'' in \emph{Artificial
  Intelligence and Statistics}, 2015, pp. 398--406.

\bibitem{zhang2017improved}
L.~Zhang, T.~Yang, J.~Yi, J.~Rong, and Z.-H. Zhou, ``Improved dynamic regret
  for non-degenerate functions,'' in \emph{Advances in Neural Information
  Processing Systems}, 2017, pp. 732--741.

\bibitem{shahrampour2018distributed}
S.~Shahrampour and A.~Jadbabaie, ``Distributed online optimization in dynamic
  environments using mirror descent,'' \emph{IEEE Transactions on Automatic
  Control}, vol.~63, no.~3, pp. 714--725, 2018.

\bibitem{nazari2019dadam}
P.~Nazari, D.~A. Tarzanagh, and G.~Michailidis, ``Dadam: A consensus-based
  distributed adaptive gradient method for online optimization,'' \emph{arXiv
  preprint arXiv:1901.09109}, 2019.

\bibitem{duchi2011adaptive}
J.~Duchi, E.~Hazan, and Y.~Singer, ``Adaptive subgradient methods for online
  learning and stochastic optimization,'' \emph{Journal of Machine Learning
  Research}, vol.~12, no. Jul, pp. 2121--2159, 2011.

\bibitem{mokhtari2016online}
A.~Mokhtari, S.~Shahrampour, A.~Jadbabaie, and A.~Ribeiro, ``Online
  optimization in dynamic environments: Improved regret rates for strongly
  convex problems,'' in \emph{2016 IEEE 55th Conference on Decision and Control
  (CDC)}.\hskip 1em plus 0.5em minus 0.4em\relax IEEE, 2016, pp. 7195--7201.

\bibitem{pu2018push}
S.~Pu, W.~Shi, J.~Xu, and A.~Nedi{\'c}, ``A push-pull gradient method for
  distributed optimization in networks,'' in \emph{2018 IEEE Conference on
  Decision and Control (CDC)}.\hskip 1em plus 0.5em minus 0.4em\relax IEEE,
  2018, pp. 3385--3390.

\bibitem{qu2018harnessing}
G.~Qu and N.~Li, ``Harnessing smoothness to accelerate distributed
  optimization,'' \emph{IEEE Transactions on Control of Network Systems},
  vol.~5, no.~3, pp. 1245--1260, 2018.

\bibitem{nedic2009distributed}
A.~Nedic and A.~Ozdaglar, ``Distributed subgradient methods for multi-agent
  optimization,'' \emph{IEEE Transactions on Automatic Control}, vol.~1,
  no.~54, pp. 48--61, 2009.

\bibitem{chiang2012online}
C.-K. Chiang, T.~Yang, C.-J. Lee, M.~Mahdavi, C.-J. Lu, R.~Jin, and S.~Zhu,
  ``Online optimization with gradual variations,'' in \emph{Conference on
  Learning Theory}, 2012, pp. 6--1.

\bibitem{johansson2008subgradient}
B.~Johansson, T.~Keviczky, M.~Johansson, and K.~H. Johansson, ``Subgradient
  methods and consensus algorithms for solving convex optimization problems,''
  in \emph{2008 47th IEEE Conference on Decision and Control}.\hskip 1em plus
  0.5em minus 0.4em\relax IEEE, 2008, pp. 4185--4190.

\end{thebibliography}
